\documentclass[a4paper]{article}
\usepackage{amssymb,amsmath,amsthm}
\usepackage[lmargin=25mm, rmargin=25mm, tmargin=25mm, bmargin=20mm]{geometry}
\usepackage{fancyhdr}
\newtheorem{thm}{Theorem}[section]
\newtheorem{prop}[thm]{Proposition}
\newtheorem{lem}[thm]{Lemma}

\newenvironment{dfn}{\medskip\refstepcounter{thm}
\noindent{\bf Definition \thesection.\arabic{thm}.}}{\medskip}

\newenvironment{note}[1][Note]{\begin{trivlist}
\item[\hskip \labelsep {\bfseries #1}]}{\end{trivlist}}

\def\R{\mathbb{R}}

\def\d{\mathrm{d}}

\DeclareMathOperator\id{id}

\DeclareMathOperator\GG{G}

\DeclareMathOperator\Ric{Ric}

\DeclareMathOperator\Ker{Ker}

\DeclareMathOperator\Vol{Vol}
\DeclareMathOperator\vol{vol}
\DeclareMathOperator\Div{div}

\DeclareMathOperator\tr{tr}
\DeclareMathOperator\Hol{Hol}

\numberwithin{equation}{section}

\begin{document}

\title{Geometric Flows of $\GG_2$ Structures}
\author{Jason D. Lotay\\ 
{\normalsize University College London}
}
\date{
}
\maketitle

\begin{abstract}
 Geometric flows have proved to be a powerful geometric analysis tool, perhaps most notably in the study of 3-manifold topology, the differentiable sphere theorem, Hermitian--Yang--Mills connections and canonical K\"ahler metrics. 
In the context of $\GG_2$ geometry, there are several geometric flows which arise. Each flow provides a potential means to study the geometry and topology associated with a given class of $\GG_2$ structures. 
We will introduce these flows, and describe some of the key known results and open problems in the field.
\end{abstract}

\tableofcontents

\section{Introduction}

Our understanding of $\GG_2$ structures, and particularly the question of when a $\GG_2$ structure can be deformed to become torsion-free, is very limited.  It is therefore useful to look to new tools to tackle open problems in the area.  An obvious avenue of attack is to use geometric flows, given their success in other geometric contexts: for example in analysing Hermitian connections (via Yang--Mills flow), convex hypersurfaces (via mean curvature flow) and perhaps most notably 3-manifolds and $\frac{1}{4}$-pinched Riemannian manifolds (via Ricci flow).   

The goal of these notes is to explain some of the basics behind the geometric flow approach to studying $\GG_2$ structures and give a brief overview of what is known.  It is important to note that several different flows of $\GG_2$ structures have been studied, based on various well-founded motivations. We shall attempt to give a brief description of each of these flows, the reasons behind them and some of the pros and cons in their study.  

As well as giving this brief survey of the landscape in geometric flows of $\GG_2$ structures, we will provide some indication of some key open questions that we believe are worthy of further exploration.

\begin{note}
These notes are based primarily on a lecture given at a Minischool on ``$\GG_2$ Manifolds and Related Topics'' at the Fields Institute, Toronto in August 2017.
\end{note}

\section{Geometric flows}

What is a geometric flow?  Informally, it is a mechanism for ``simplifying'' or ``decomposing'' a given geometric structure into one or several ``canonical'' or ``special'' pieces.  
Thus, the primary goals of geometric flows are to show the existence of special geometric objects and to determine which geometric objects can be deformed to special ones.  By answering these questions, one can then hope to 
understand large classes of geometric structures by understanding the much smaller class of canonical ones.

\subsection{Heat flow}

\paragraph{Functions.} The motivation for geometric flows comes from the heat flow for functions $f$ on a Riemannian manifold:
\begin{equation}\label{eq:heat}
 \left(\frac{\partial}{\partial t}+\Delta\right)f=0.
\end{equation}
(Here, and throughout, we will use the geometer's convention that the Laplacian $\Delta$ is a non-negative operator, so $\Delta=\d^*\d$ on functions.) The heat flow is parabolic, which means that if we consider \eqref{eq:heat} on a compact manifold $M$, then 
a short time solution to \eqref{eq:heat} is guaranteed to exist and the equation is ``regularizing'' (a notion we shall clarify in a moment).

We now make an elementary but fundamental observation. 

\begin{prop}
The heat flow is the negative gradient flow for the Dirichlet energy:
\begin{equation}\label{eq:Dirichlet}
 \frac{1}{2}\int_M|\d f|^2\vol_M\geq 0.
\end{equation}
\end{prop}

\begin{proof}
We see that for any $t$-dependent family of functions we have
\begin{equation}
 \frac{\partial}{\partial t}\frac{1}{2}\int_M|\d f|^2\vol_M=\langle \frac{\partial}{\partial t}\d f,\d f\rangle_{L^2}=\langle\d\frac{\partial}{\partial t}f,\d f\rangle_{L^2}=\langle\frac{\partial f}{\partial t},\Delta f\rangle_{L^2}.
\end{equation}
\end{proof}

\noindent Thus \eqref{eq:Dirichlet} will decrease fastest along the heat flow \eqref{eq:heat}, and the critical points for the Dirichlet energy (which are exactly the stationary points for \eqref{eq:heat}) are given by the constant functions (i.e.~$\d f=0$) which are the absolute minimizers for the energy.  

From the gradient flow point of view we should expect that given any smooth function, by following the heat flow we should be able to deform it into a critical point for \eqref{eq:Dirichlet}, i.e.~a constant function.  
We now show that this is the case.

\begin{thm}\label{thm:heat}
Suppose that $f=f(x,t)$ solves the heat equation \eqref{eq:heat} on a compact manifold $M$ and $f(x,0)$ is smooth.  Then $f(x,t)$ exists for all $t>0$ and $f(x,t)\to c\in\R$ smoothly as $t\to\infty$, where   
$$c=\frac{1}{\Vol(M)}\int_M f(x,0)\vol_M.$$
\end{thm}

\begin{proof}  If we suppose that $f$ is smooth at $t=0$, then it is smooth for all $t>0$ as well and 
the eigenfunctions of $\Delta$ span $L^2(M)$, so at each time $t$ we can write 
\begin{equation}\label{eq:f.decomp}
 f(x,t)=\sum_{\lambda}c_\lambda(t) f_\lambda(x)
\end{equation}
for functions $c_\lambda$ of time $t$ and $f_{\lambda}$ on $M$, where $\Delta f_\lambda=\lambda f_\lambda$ for $\lambda\geq 0$ and the $f_\lambda$ form a complete orthonormal system for $L^2(M)$. 
  It quickly follows from inserting \eqref{eq:f.decomp} 
in \eqref{eq:heat} that
\begin{equation}\label{eq:decay}
c_{\lambda}(t)=c_\lambda(0)e^{-\lambda t},
\end{equation}
and so the solution of \eqref{eq:heat} actually exists for all time $t> 0$.  Moreover,  the solution converges as $t\to\infty$ to
\begin{equation}
c_0f_0=\frac{1}{\Vol(M)}\int_M f(x,0) \vol_M, 
\end{equation}
the ``average value'' of $f$ at time $0$. 
\end{proof}

 Thus, the heat flow ``regularizes'' the function $f$ in that it simplifies it as much as possible (it turns it into a constant) and we see that the higher the frequency (i.e.~eigenvalue) 
of the eigenfunction of $\Delta$ in the expansion \eqref{eq:f.decomp}, the faster that component of $f$ decays under the flow by \eqref{eq:decay}.  The $f_\lambda$ for high $\lambda$ correspond  to higher ``oscillations'' of $f$, 
and so these ``wiggles'' in $f$ get smoothed out by \eqref{eq:heat}, eventually giving a constant.   In terms of the Dirichlet energy functional \eqref{eq:Dirichlet}, 
it shows that every function can be deformed to a minimizer (so the space of smooth functions retracts onto the constant functions, which are the critical points of the functional), 
and the minimizer we find is determined by the average value of $f$.

  Notice that our analysis in the proof of Theorem \ref{thm:heat} implies the following.
    \begin{lem}
  The integral of $f$ is constant along \eqref{eq:heat}.
  \end{lem}  
  \noindent  We are therefore free to modify \eqref{eq:heat} and consider
\begin{equation}\label{eq:heat.modified}
 \left(\frac{\partial}{\partial t}+\Delta-\lambda_1\right)f=0,
\end{equation}
for functions $f$ with $\int_M f=0$, where $\lambda_1$ is the first positive eigenvalue of $\Delta$ on $M$.  It is easy to see that \eqref{eq:heat.modified} is still parabolic and that if $\int_Mf=0$ initially then it stays zero for all $t$ under \eqref{eq:heat.modified}.  However, under 
\eqref{eq:heat.modified}, we see that the flow no longer converges to a constant, but instead to the projection of $f$ to the $\lambda_1$-eigenspace of $\Delta$ (which may now have several components if $\lambda_1$ is not
a simple eigenvalue).  Again, this flow ``regularizes'' $f$, throwing away all of the higher eigenmodes of $\Delta$ in the limit. 

\paragraph{Forms.} We can also consider the heat flow on differential $k$-forms $\alpha$ on a compact manifold $M$:
\begin{equation}\label{eq:heat.forms}
\left(\frac{\partial}{\partial t}+\Delta\right)\alpha=0,
\end{equation}
where $\Delta$ is the Hodge Laplacian
\begin{equation}
\Delta=\d\d^*+\d^*\d.
\end{equation}
The flow \eqref{eq:heat.forms} is now the gradient flow for the Dirichlet energy
\begin{equation}\label{eq:Dirichlet.forms}
\frac{1}{2}\int_M|\d\alpha|^2+|\d^*\alpha|\vol_M\geq 0
\end{equation}
by a similar argument as before.   Again decomposing $\alpha(t)$ at each time $t$ using eigenforms for $\Delta$, we have the following.

\begin{thm} The heat equation \eqref{eq:heat.forms} for $\alpha(t)$ on a compact manifold starting at a smooth form $\alpha(0)$ exists for all time and converges smoothly to the projection of $\alpha(0)$ to the $0$-eigenforms for $\Delta$, i.e.~the harmonic $k$-forms
\begin{equation}\label{eq:harmonic.forms}
\d\alpha=\d^*\alpha=0.
\end{equation}
\end{thm}
\noindent The harmonic forms are precisely the critical points of \eqref{eq:Dirichlet.forms} and are clearly absolute minimizers as they are zeros for the energy functional.  

Now, the harmonic forms are only a finite-dimensional space in the space of $k$-forms, so given any initial $k$-form it could well be that the heat flow will just send it to zero, which is clearly a legitimate critical point for the flow (though not an interesting one!).  For example if $\alpha(0)$ is exact or coexact (or the sum of forms of this type), the heat flow will just go to $0$.  

\medskip

To ensure that we find a non-trivial critical point, we could restrict attention to closed $k$-forms: 
\begin{equation}
\d\alpha=0.
\end{equation}
Notice that this is preserved by \eqref{eq:heat.forms} since in this case we have
\begin{equation}\label{eq:heat.closed.forms}
\frac{\partial}{\partial t}\alpha=-\Delta\alpha=-(\d\d^*+\d^*\d)\alpha=-\d\d^*\alpha,
\end{equation}
so in fact we have that $\alpha(t)$ lies in the fixed cohomology class $[\alpha(0)]$ for all time as the right-hand side of \eqref{eq:heat.closed.forms} is exact.  Therefore, if we have that $[\alpha(0)]\neq 0$ is a non-trivial cohomology class, we know that \eqref{eq:heat.closed.forms} will exist for all time and converge to the non-zero harmonic representative of that class (which we know exists and is unique by Hodge theory).  

We could also equally well have restricted to coclosed $k$-forms
\begin{equation}
\d^*\alpha=0,
\end{equation}  
since this is also preserved by a similar argument. This time $*\alpha(t)$ will lie in the fixed cohomology class $[*\alpha(0)]$ for all time and the flow will converge to the Hodge dual of the harmonic representative of $[*\alpha(0)]$.  

We summarize these findings.

\begin{prop}
Suppose that $\alpha(t)$ is a family of $k$-forms on a compact $n$-manifold $M$  solving \eqref{eq:heat.forms}.  
\begin{itemize}
\item[\emph{(a)}]  If  $\d\alpha(0)=0$, $\alpha(t)$ exists for all $t>0$ satisfying $\d\alpha(t)=0$ for all $t$ and converges smoothly to the unique harmonic $k$-form in $[\alpha(0)]\in H^k(M)$.  
\item[\emph{(b)}] If $\d^*\alpha(0)=0$, $\alpha(t)$ exists for all $t>0$ satisfying $\d^*\alpha(t)=0$ for all $t$ and and converges smoothly to the Hodge dual of the unique harmonic $(n-k)$-form in $[*\alpha(0)]\in H^{n-k}(M)$. 
\end{itemize}
\end{prop}

We might hope, at least naively, that we could have similar good behaviour in geometric flows as to the heat flow, and thus obtain ways to canonically represent classes of geometric structures, just as harmonic forms uniquely represent all cohomology classes.

\subsection{Ricci flow and mean curvature flow}

Geometric flows aim to act on the same principle as the heat flow, two canonical examples being Ricci flow on metrics $g$ and mean curvature flow on immersions $F$ into a Riemannian manifold:
\begin{equation}\label{eq:Ricci.MCF}
 \frac{\partial}{\partial t}g=-2\Ric(g)\quad\text{and}\quad \frac{\partial}{\partial t}F=H,
\end{equation} 
where $\Ric(g)$ denotes the Ricci curvature tensor of $g$ and $H$ denotes the mean curvature vector of the immersion $F$.  
(Two other key examples of geometric flows of significant interest where many results have been obtained are the harmonic map heat flow and Yang--Mills flow, but we do not discuss them here.) 
 Under suitable choices of coordinates, \eqref{eq:Ricci.MCF} can be seen as ``heat flows'', however this time the Laplacian depends on the metric or immersion respectively, and so the flows are nonlinear.

\paragraph{Parabolicity.}  The flows \eqref{eq:Ricci.MCF} are not parabolic due to geometric invariance in the problem: in Ricci flow this is diffeomorphism invariance, and in mean curvature flow this is invariance under reparametrisation.  However, once one 
kills this geometric invariance, one obtains a parabolic equation.   

For example, in Ricci flow, one can apply the so-called DeTurck's trick:
\begin{equation}\label{eq:RicciDT}
 \frac{\partial}{\partial t}h=-2\Ric(h)+\mathcal{L}_{X(h)}h,
\end{equation}
where $X(h)$ is a suitably chosen vector field (depending on the metric $h$) which ensures that \eqref{eq:RicciDT} is parabolic and so has a short time solution which is regularizing.  We can get a solution to Ricci flow from a solution to \eqref{eq:RicciDT} by considering $g=\Phi^*h$, where $\Phi$ are diffeomorphisms defined by
\begin{equation}\label{eq:DT}
 \frac{\partial}{\partial t}\Phi=-X(h) \quad\text{and} \quad \Phi(0)=\id.
\end{equation}

\begin{prop}\label{prop.RicciDT}
Suppose that $h$ are metrics satisfying \eqref{eq:RicciDT} and $\Phi$ are diffeomorphisms satisfying \eqref{eq:DT}.  Then $g=\Phi^*h$ satisfies the Ricci flow in \eqref{eq:Ricci.MCF}.
\end{prop}

\begin{proof}
By \eqref{eq:RicciDT} and \eqref{eq:DT},
\begin{equation}\label{eq:RFDTHM}
 \frac{\partial}{\partial t}g=\frac{\partial}{\partial t}\Phi^*h=\Phi^*\frac{\partial}{\partial t}h-\Phi^*\mathcal{L}_{X(h)}h=-2\Phi^*\Ric(h)=-2\Ric(g).
\end{equation}
\end{proof}

This result is great but, it is natural to ask: what is a good choice of $X(h)$?  The idea is, given $h$, for any symmetric 2-tensor $k$ to consider the ``gravitational tensor''
\begin{equation}\label{eq:grav}
G(k)=k-\frac{1}{2}(\tr k)h,
\end{equation}
whose divergence is given by the 1-form:
\begin{equation}\label{eq:div.grav}
\Div G(k)=\Div(k)+\frac{1}{2}\d(\tr k).
\end{equation}
(Here, by the divergence we mean the formal adjoint of the map $X^{\flat}\mapsto \frac{1}{2}\mathcal{L}_{X}h$, so that $\Div(k)$ is the negative of the trace on the first two indices of $\nabla k$; i.e.~$\Div(k)_j=-\nabla_i k_{ij}$.  The musical isomorphisms, $\nabla$ and trace are all defined by $h$.) 
If $k$ is a \emph{fixed} Riemannian metric then, using $h$, we can view $k$ as an invertible map on 1-forms and so 
\begin{equation}\label{eq:DT.vfield}
X(h)=(k^{-1}\Div G(k))^{\sharp},
\end{equation}
where the musical isomorphism is again given by $h$, is a well-defined vector field.  

\begin{thm}
If we choose the vector field $X(h)$ as in \eqref{eq:DT.vfield}, the 
Ricci--DeTurck flow \eqref{eq:RicciDT} is parabolic.
\end{thm}

\begin{proof}
It is straightforward to compute that, with the choice of $X(h)$ in \eqref{eq:DT.vfield}, the linearisation of \eqref{eq:RicciDT} is simply
\begin{equation}
\frac{\partial}{\partial t}h=-\Delta h,
\end{equation}
the heat equation on $h$.
\end{proof}

\noindent Moreover, \eqref{eq:DT} is harmonic map flow, which is parabolic, so we can solve \eqref{eq:RicciDT} and \eqref{eq:DT} uniquely for short time by parabolic theory, and hence obtain a unique short time solution to the Ricci flow by Proposition \ref{prop.RicciDT}.

\medskip

Now, by analogy with the heat flow, in Ricci flow the ``eigenmodes'' we need to consider are solutions to 
\begin{equation}
 \Ric(g)=\lambda g;
\end{equation}
in other words, Einstein metrics.  Here the ``eigenvalues'' $\lambda$ have no distinguished sign and so, by analogy with the heat flow analysis above, we cannot say what will happen along the flow in general.  However, if the 
flow exists for all time and converges, then it must tend to a Ricci-flat metric - the ``zero mode'' in the expansion of $g$ in Einstein metrics, if you will.  Similarly, for mean curvature flow, if the flow exists for all time and 
converges, we would obtain a minimal immersion (and the ``expansion'' of the immersion should be into constant mean curvature immersions).  We see in both cases that we are breaking up our geometric object into pieces of 
significant interest.  Moreover, the long-time existence and convergence of the flow allows us both to find the special object (a Ricci-flat metric or minimal immersion) and, at the same time, show that our initial geometric 
object can be deformed smoothly into the special one, which again is an important and challenging problem to solve.

\paragraph{Compact surfaces.} To see the power of geometric flows it is instructive to look at Ricci flow in dimension 2.  Here, Ricci curvature is just the Gauss curvature of the surface (up to a multiple) and there are three possibilities on a compact 
orientable surface.
\begin{itemize}
\item The flow exists for all time and converges.  This means that the surface has a flat metric, and so must have genus 1 (by Gauss--Bonnet).
 \item The flow exists for all time but does not converge.  In this case, just as when we perturbed the heat flow in \eqref{eq:heat.modified},  we can modify the Ricci flow and show that this modified flow exists for all time and converges to a hyperbolic metric.  Thus the 
surface must have genus at least 2.
\item The flow exists for only a finite time.  This does not have a direct heat flow analogue, but one can again modify the Ricci flow as in the previous case (now by adding a term with the opposite sign), and show that this converges now to a constant positive 
curvature metric, which means that the surface must be a sphere.  This is the difficult case in the analysis of the Ricci flow and this is typical of geometric flows: the case corresponding to ``negative eigenvalues''
(which do not happen for the Laplacian in the heat flow) is the most challenging to understand.
\end{itemize}
Thus, the Ricci flow in dimension 2 gives an alternative means to prove the uniformization theorem.  In particular, 
that constant curvature metrics exist on any compact orientable surface, and the topology of the surface uniquely determines the 
sign of the constant curvature.

\paragraph{Gradient flow.} Finally, one can also interpret \eqref{eq:Ricci.MCF} as gradient flows: in the case of mean curvature flow this is nothing but the negative gradient flow of the volume functional on immersions, but for Ricci flow the gradient flow
interpretation is more subtle and involved so we shall not describe it here. 

 Needless to say, the fact that they are gradient flows is very helpful, since then one has some expectation of what one might hope to happen along the flow, 
as one has a monotone quantity (the analogue of the Dirichlet energy) along the flow which is trying to reach a critical value for the functional.  

 However, even with the gradient flow point-of-view, the nonlinearity of the problem and the potential complexity of the 
topology of the space of geometric objects we are considering means that we cannot always hope for the analysis of our flow to be straightforward and to go as expected.  For example, the Ricci flow and mean curvature flow 
have special features (both good and bad) due to nonlinearity which simply cannot possibly occur in the standard heat flow.

\subsection{Singularities}

A singularity in a geometric flow is a point where the flow cannot be continued, because some quantity blows up to infinity.  We already saw this in the Ricci flow in dimension 2, where there is always a singularity in finite time if we work on a sphere.  Singularities may sound bad, and 
they definitely can be, but they can also be very helpful because they may tell you that you need to break up your geometric object into several pieces to get canonical objects.  This happens for example in 3-dimensional Ricci flow, where singularities can be used to decide how to break up the 3-manifold according to Thurston's Geometrization Conjecture (now a theorem by Perelman's work).

The question is: what happens at a singularity geometrically?  In good situations the singularity will be modelled on a special solution to the flow called a soliton.  By ``modelled on'' we mean that by appropriately rescaling the flow around the singular point, both in space and time, in the limit we should see a soliton.  

\begin{dfn}
A soliton is a solution to the flow which is ``self-similar'', meaning that it moves very simply under the flow, just under rigid motions and diffeomorphisms or reparametrisations (or whatever notion of invariance is present in the problem).  

Solitons which just move under diffeomorphisms are called steady, those which rescale getting smaller are called shrinking, and those which rescale getting larger are called expanding.  

That is usually all of them (like in Ricci flow), but in mean curvature flow a soliton can also just translate, which is, rather unimaginatively, called a translating soliton.  
\end{dfn}

Simple examples of solitons in Ricci flow are given by constant curvature metrics: flat metrics are steady (in fact critical points), constant positive curvature metrics are shrinking (like standard round spheres) and constant negative curvature metrics are expanding (like hyperbolic space).  

From this point of view we see that shrinking solitons are the ones we should be most concerned with, since they will become singular in finite time, shrinking away.  However, steady and expanding solitons also play an important role. 

Steady solitons are potentially where the flow can get ``stuck'' going round and round under diffeomorphisms and never converging.  Non-stationary compact examples of steady solitons are ruled out if you have a standard gradient flow interpretation of the flow,  since the corresponding functional would be constant for steady solitons, which is a contradiction unless they are stationary.  This shows one of the benefits of knowing that your geometric flow is the gradient flow of some functional.  

On the other hand, expanding solitons give a potential mechanism to escape from a singularity, since they expand away from a singular geometric object.

It is therefore clear that understanding singularities and solitons is an important part of the study of any geometric flow.  

\section{$\GG_2$ structures}

Given this discussion of geometric flows, we are now motivated to ask the question: are there (useful) geometric flows of $\GG_2$ structures on a (compact) 7-manifold $M$ 
and what do they want to achieve?  We should perhaps not expect there to be just one useful 
flow to consider: for immersions, both mean curvature flow and inverse mean curvature flow have important geometric uses, for example.  We therefore need to think about what are the important classes of $\GG_2$ structures 
that we want to analyse and what we expect to be ``canonical'' representatives for these classes.  (For details about $\GG_2$ structures, which are equivalent to positive 3-forms, see for example \cite{Joyce}.)

\subsection{Torsion-free and torsion forms}

Clearly the most important class of $\GG_2$ structures are the torsion-free ones, given by positive 3-forms $\varphi$ on $M$ satisfying
\begin{equation}\label{eq:torsion.free}
 \nabla_{\varphi}\varphi=0\quad\Leftrightarrow\quad \d\varphi=\d^*_{\varphi}\varphi=0\quad\Leftrightarrow\quad \Hol(g_{\varphi})\subseteq\GG_2.
\end{equation}
(We are being slightly sloppy in the last equivalence, since given a metric there are infinitely many $\GG_2$ structures inducing that metric, 
so we mean that the holonomy $\Hol(g_{\varphi})$ of $g_{\varphi}$ is contained in $\GG_2$ if and only if there is some $\GG_2$ structure $\varphi$ inducing $g_{\varphi}$ which is closed and coclosed.)
We also know that we may equivalently define $\GG_2$ structures on oriented, spin, Riemannian 7-manifolds using unit spinors $\sigma$, and the condition for the $\GG_2$ structure to be torsion-free is that $\sigma$ is parallel with respect to the spin connection
:
\begin{equation}
 \nabla 
 \sigma=0.
\end{equation}
(We can recall the relationship between unit spinors and positive 3-forms:
\begin{equation}
 4\sigma\otimes\sigma=1+\varphi+*_{\varphi}\varphi+\vol_\varphi,
\end{equation}
up to appropriate normalizations and sign conventions.)

\paragraph{Closed and coclosed $\GG_2$ structures.} However, there are other obvious (potentially) important classes of $\GG_2$ structures, for example closed $\GG_2$ structures
\begin{equation}\label{eq:closed}
 \d\varphi=0,
\end{equation}
or coclosed $\GG_2$ structures
\begin{equation}\label{eq:coclosed}
 \d^*_{\varphi}\varphi=0.
\end{equation}
It is worth noting that, on the face of it, \eqref{eq:closed} is much stronger than \eqref{eq:coclosed}: the first is a condition on a 4-form in 7-dimensions (so 35 equations at each point), whereas the second is a condition on a 5-form in 7-dimensions (so 21 equations at each point).

Both conditions \eqref{eq:closed} and \eqref{eq:coclosed} can be satisfied independently on any open 7-manifold admitting a $\GG_2$ structure by a straightforward h-principle argument; thus one can say (in some sense) that these conditions are only truly meaningful on compact 7-manifolds.    In fact, \eqref{eq:coclosed} can always be satisfied on a compact 7-manifold admitting a $\GG_2$ structure also by an h-principle \cite{CrowNord}, but it is currently unknown whether the same is true for condition \eqref{eq:closed} or not: this again reflects the fact that \eqref{eq:closed} is a stronger condition than \eqref{eq:coclosed}.  

\begin{thm}\label{thm:hprinciple}
Let $\varphi$ be a $\GG_2$ structure on $M$.
\begin{itemize}
\item[\emph{(a)}] If $M$ is open, then there exists a $\GG_2$ structure $\tilde{\varphi}$ $C^0$-close to $\varphi$ satisfying \eqref{eq:closed}.
\item[\emph{(b)}] If $M$ is either open or compact, then there exists a $\GG_2$ structure $\tilde{\varphi}$ $C^0$-close to $\varphi$ satisfying \eqref{eq:coclosed}. 
\end{itemize}
\end{thm}

One can interpret the h-principle result for coclosed $\GG_2$ structures as both positive and negative.  On the one hand, it is good because we can always assume the condition \eqref{eq:coclosed} holds for $\varphi$ if we want, though we have very little control on the 
$\varphi$ produced by the h-principle: we can assume it is $C^0$-close to our original $\GG_2$ structure but the method only produces $\varphi$ which will be very far away in the $C^1$-topology.  On the other hand, it says that the condition \eqref{eq:coclosed} is, in a sense, meaningless and that talking about coclosed $\GG_2$ structures is the same as talking about all $\GG_2$ structures, which becomes a topological rather than a geometric question.  

\paragraph{Torsion forms.} One could also conceivably look at other special torsion classes by setting various combinations of the intrinsic torsion forms to vanish, recalling that these are given by
\begin{equation}\label{eq:torsion}
 \d\varphi=\tau_0*\varphi+3\tau_1\wedge\varphi+*_{\varphi}\tau_3\quad\text{and}\quad \d*_{\varphi}\varphi=4\tau_1\wedge*_{\varphi}\varphi+\tau_2\wedge\varphi,
\end{equation}
where $\tau_0\in C^{\infty}(M)$, $\tau_1\in \Omega^1(M)$, $\tau_2\in\Omega^2_{14}(M)$, $\tau_3\in\Omega^3_{27}(M)$, with the standard notation referring to the type decomposition of forms determined by $\varphi$ (see, for example, \cite{Bryant}).  Recall that 
$\beta\in\Omega^2_{14}(M)$ if and only if
\begin{equation}\label{eq:omega214}
\beta\wedge\varphi=-*_{\varphi}\beta\quad\Leftrightarrow\quad \beta\wedge*_{\varphi}\varphi=0,
\end{equation}
and that $\gamma\in\Omega^3_{27}(M)$ if and only if
\begin{equation}
 \gamma\wedge\varphi=\gamma\wedge*_{\varphi}\varphi=0.
\end{equation}
Moreover, recall that we have an isomorphism $i_{\varphi}:S^2T^*M= \text{Span}\{g_{\varphi}\}\oplus S^2_0T^*M\to \Lambda^3_1\oplus\Lambda^3_{27}$ given on decomposable elements $\alpha\circ\beta$ by
$$i_{\varphi}(\alpha\circ\beta)=\alpha\wedge*_{\varphi}(\beta\wedge *_{\varphi}\varphi)+\beta\wedge*_{\varphi}(\alpha\wedge*_{\varphi}\varphi).$$
We also have an explicit way to invert $i_{\varphi}$ using $j_{\varphi}:\Lambda^3T^*M\to S^2T^*M$ given by
$$j_{\varphi}(\gamma)(u,v)=*_{\varphi}(u\lrcorner\varphi\wedge v\lrcorner\varphi\wedge\gamma).$$
Notice that $i_{\varphi}(g_{\varphi})=6\varphi$, $j_{\varphi}(\varphi)=6g_{\varphi}$ and $\Ker j_{\varphi}=\Lambda^3_7$.

\paragraph{Other classes.} A particular class of $\GG_2$ structures one could consider are the nearly parallel $\GG_2$ structures
\begin{equation}\label{eq:nearly.parallel}
 \d\varphi=\tau_0*_{\varphi}\varphi
\end{equation}
for a constant $\tau_0$.  These structures define Einstein metrics with non-negative scalar curvature, and so there is a potential relation between these structures and our discussion of the Ricci flow above.     

One could also view matters in terms of spinors, and study geometric flows of unit spinors.  
One can then try studying parallel spinors or, more generally, Killing spinors, as well as other special types of spinors (e.g.~twistor spinors).

\subsection{General flows}

Based on this discussion, it is clear that there are many possible geometric flows one could write down, and each one could potentially tackle different open problems in $\GG_2$ geometry.

That said, one can describe how various key quantities vary under a general flow of $\GG_2$ structures (see \cite{Bryant,Karigiannis}).  
By the type decomposition of 3-forms, any geometric flow of $\GG_2$ structures can be written
\begin{equation}\label{eq:general.flow}
 \frac{\partial}{\partial t}\varphi=3f_0\varphi+*_{\varphi}(f_1\wedge\varphi)+f_3=i_{\varphi}(h)+X\lrcorner*_{\varphi}\varphi
\end{equation}
where $f_0\in C^{\infty}(M)$, $f_1\in\Omega^1(M)$ and $f_3\in\Omega^3_{27}(M)$, $h\in C^{\infty}(S^2T^*M)$ and $X\in C^{\infty}(TM)$ at each time $t$.  From this one can see that the metric and 4-form evolve as follows. 

\begin{prop} Along \eqref{eq:general.flow} we have
\begin{equation}\label{eq:metric}
 \frac{\partial}{\partial t} g_{\varphi}=2f_0g+\frac{1}{2}j_{\varphi}(f_3)=2h
\end{equation}
and
\begin{equation}\label{eq:starphi}
 \frac{\partial}{\partial t} *_{\varphi}\!\varphi=4f_0*\varphi+f_1\wedge\varphi-*_{\varphi}f_3.
\end{equation}
\end{prop}

In particular, along any flow \eqref{eq:general.flow}, the evolution of the metric is independent of the vector field $X$, and the volume form evolves as follows.

\begin{prop}
Along \eqref{eq:general.flow} we have
\begin{equation}\label{eq:vol.flow}
 \frac{\partial}{\partial t}\vol_\varphi=7f_0\vol_{\varphi}=\frac{1}{3}\frac{\partial}{\partial t}\varphi\wedge*_{\varphi}\varphi.
\end{equation}
\end{prop}

This formula is useful to study Hitchin's volume functional \cite{Hitchin} on a compact manifold $M$.  (Note that the published version \cite{Hitchin2} of \cite{Hitchin} omits the material on $\GG_2$ structures.)

\begin{prop} Along \eqref{eq:general.flow}, the volume functional
\begin{equation}\label{eq:Vol}
 \Vol(\varphi)=\frac{1}{7}\int_M\varphi\wedge *\varphi=\int_M\vol_{\varphi}=\Vol(M,g_{\varphi}).
\end{equation}
satisfies (by \eqref{eq:vol.flow})
\begin{equation}\label{eq:Vol.flow}
\frac{\partial}{\partial t}\Vol(\varphi)=7\int_Mf_0\vol_{\varphi}=\frac{1}{3}\langle\frac{\partial}{\partial t}\varphi,\varphi\rangle_{L^2}.
\end{equation}
\end{prop}

This shows in particular that $\Vol(\varphi)$ will be monotone along any flow for which $f_0$ has a sign, and the critical points of the functional will be characterised by the $\GG_2$ structures for which $f_0=0$.  It also shows that the obvious gradient flow for $\Vol(\varphi)$ is
\begin{equation}
\frac{\partial}{\partial t}\varphi=\lambda\varphi
\end{equation}
for some $\lambda>0$ (or $\lambda<0$ for the negative gradient flow).  
This is clearly useless since all it does is rescale $\varphi$!  Therefore, if one wants to think about making use of the volume functional for a gradient flow, we should consider restricting the class of $\GG_2$ structures we work with.

\section{Laplacian flow}

The geometric flow of $\GG_2$ structures that has received the most attention is the Laplacian flow due to Bryant \cite{Bryant}.

\begin{dfn} The Laplacian flow is given by
\begin{equation}\label{eq:Lflow}
 \frac{\partial}{\partial t}\varphi=\Delta_{\varphi}\varphi,
\end{equation}
where
\begin{equation}
 \Delta_\varphi=\d\d^*_{\varphi}+\d^*_{\varphi}\d
\end{equation}
is the Hodge Laplacian.  
\end{dfn}

On a compact manifold, we see that 
\begin{equation}
\Delta_{\varphi}\varphi=0\quad\Leftrightarrow\quad \d\varphi=\d^*_{\varphi}\varphi=0
\end{equation}
by integration by parts, and so torsion-free $\GG_2$ structures will be the critical points of \eqref{eq:Lflow}.  

\subsection{Closed $\GG_2$ structures}

Bryant's suggestion is to restrict \eqref{eq:Lflow} to closed $\GG_2$ structures $\varphi$ as in \eqref{eq:closed}.  A key motivation is the usual one in $\GG_2$ geometry: namely that the torsion-free condition naturally splits into a linear condition \eqref{eq:closed} and a nonlinear condition \eqref{eq:coclosed}.  Thus, it is useful to assume the linear condition is satisfied then try to solve the nonlinear one.  This strategy is the only one that has proved to be successful, by the work of Joyce \cite[Chapter 11]{Joyce}.  Hence, it clearly makes sense to follow the same approach in a geometric flow.

It will turn out that when \eqref{eq:Lflow} exists the closed condition is preserved.  In that case
\begin{equation}
 \frac{\partial}{\partial t} \varphi=\d\d^*_{\varphi}\varphi,
\end{equation}
so in fact \eqref{eq:Lflow} stays within a fixed cohomology class $[\varphi(0)]$.  

\begin{prop}\label{prop:Lflow.closed}
If $\varphi(t)$ satisfies \eqref{eq:Lflow} on $M$  and 
$\d\varphi(t)=0$ then $\varphi(t)\in [\varphi(0)]\in H^3(M)$ for all $t$ for which the flow exists.
\end{prop}

\noindent This is perhaps reminiscent of the K\"ahler--Ricci flow on a manifold with $c_1=0$, 
which starts with a K\"ahler form and stays within the K\"ahler class, but it is not clear at all whether such an analogy is pertinent or a red herring.

When we restrict to closed $\GG_2$ structures we can decompose 
\begin{equation}\label{eq:Ldecomp}
 \Delta_{\varphi}\varphi=\frac{1}{7}|\tau_2|^2\varphi+f_3.
\end{equation}
We can see this because
\begin{equation}\label{eq:tau2}
\d\!*_\varphi\!\varphi=\tau_2\wedge\varphi=-*_\varphi\!\tau_2
\end{equation}
by \eqref{eq:torsion} and \eqref{eq:omega214}, and therefore
\begin{equation}
 \Delta_{\varphi}\varphi=\d\d^*_{\varphi}\varphi=\d\tau_2
\end{equation}
and differentiating \eqref{eq:tau2} gives
\begin{equation}
 0=\d\tau_2\wedge\varphi=\Delta_{\varphi}\varphi\wedge\varphi=0.
\end{equation}
This means that $f_1=0$ in \eqref{eq:general.flow}.  Moreover, \eqref{eq:tau2} and \eqref{eq:omega214} imply that
\begin{equation}
 \d\tau_2\wedge*_\varphi\varphi=\d(\tau_2\wedge*_\varphi\varphi)-\tau_2\wedge \d*_\varphi\varphi=-\tau_2\wedge\tau_2\wedge\varphi=|\tau_2|^2\vol_{\varphi}.
\end{equation}
Hence, for closed $\GG_2$ structures, \eqref{eq:Ldecomp} implies the following.
\begin{lem}
For a closed $\GG_2$ structure $\varphi$, we have
\begin{equation}
 \Delta_{\varphi}\varphi=0\quad\Leftrightarrow\quad \d^*_\varphi\varphi=0.
\end{equation}
\end{lem}
Therefore, the critical points of \eqref{eq:Lflow} on closed $\GG_2$ structures are precisely the torsion-free $\GG_2$ structures, without assuming compactness.  This might seem like a minor point (since we will mainly only care about the Laplacian flow on compact manifolds) but it appears to hint at the special character of the Laplacian flow because it is restricted to closed $\GG_2$ structures.

Moreover, we can show that 
\begin{equation}\label{eq:Delta.metric}
\Delta_{\varphi}\varphi=i_{\varphi}\Big(-\Ric(g_{\varphi})+\frac{4}{21}|\tau_2|^2g_{\varphi}+\frac{1}{8}j_{\varphi}\big(*_{\varphi}(\tau_2\wedge\tau_2)\big)\Big)
\end{equation}
so that by \eqref{eq:metric} we have the following.

\begin{prop} Along \eqref{eq:Lflow} we have
\begin{equation}\label{eq:Lflow.metric}
\frac{\partial}{\partial t}g_{\varphi}=-2\Ric(g_{\varphi})+\frac{8}{21}|\tau_2|^2g_{\varphi}+\frac{1}{4}j_{\varphi}\big(*_{\varphi}(\tau_2\wedge\tau_2)\big).
\end{equation}
\end{prop}

\subsection{Volume functional}

\begin{lem}\label{lem.Vol}
Along \eqref{eq:Lflow} for closed $\GG_2$ structures, the volume functional $\Vol(\varphi)$ is monotonically increasing.
\end{lem}

\begin{proof}
 We can already see from \eqref{eq:Ldecomp} that, in terms of \eqref{eq:general.flow}, 
\begin{equation}\label{eq:f0.Lflow}
f_0=\frac{1}{7}|\tau_2|^2\geq 0.
\end{equation}
Hence, by \eqref{eq:Vol.flow}, the volume functional $\Vol(\varphi)$ is monotonically increasing along the Laplacian flow.  
\end{proof}

 This allows us to give a very quick alternative proof of the main result of \cite{Lin}.

\begin{prop}\label{prop:no.solitons}
Let $\varphi(t)$ be a steady or shrinking soliton to the Laplacian flow \eqref{eq:Lflow} for closed $\GG_2$ structures on a compact manifold.  Then $\varphi(t)$ is stationary, i.e.~$\varphi(t)=\varphi(0)$ is torsion-free for all $t$.
\end{prop}

\begin{proof}
If $\varphi(t)$ is steady or shrinking then 
\begin{equation}
\frac{\partial}{\partial t}\Vol(\varphi(t))\leq 0.
\end{equation}
Therefore, by Lemma \ref{lem.Vol}, $\Vol(\varphi(t))$ is constant.  Hence, by \eqref{eq:Vol.flow} and \eqref{eq:f0.Lflow}, we must have that 
 \begin{equation}
 f_0=\frac{1}{7}|\tau_2|^2=0,
 \end{equation}
 which means $\tau_2=0$, as required.
\end{proof}
\noindent It is important to note that this result fails in the non-compact setting: Lauret \cite{Lauret2} has constructed examples of non-compact shrinking and steady solitons which are not stationary.  The proof of Proposition \ref{prop:no.solitons} is not valid here since the volume is not well-defined for these non-compact examples.

We now show that the Laplacian flow for closed $\GG_2$ structures is actually the gradient flow for $\Vol(\varphi)$ where, from now on, we only consider the volume functional restricted to a given cohomology class.

\begin{prop}\label{prop:Lflow.gradientflow}  The Laplacian flow \eqref{eq:Lflow} for closed $\GG_2$ structures $\varphi$ is the gradient flow for the volume functional $\Vol(\varphi)$ in \eqref{eq:Vol} restricted to $[\varphi]$.
\end{prop}

\begin{proof}
We know that the flow stays within a given cohomology class, so we can write
\begin{equation}
 \varphi(t)=\varphi(0)+\d\eta(t)
\end{equation}
for some 2-forms $\eta$ and the Laplacian flow is really, in some sense, a flow on 2-forms.  (Again, this is reminiscent of K\"ahler--Ricci flow with $c_1=0$ as the flow becomes a flow on K\"ahler potentials, but this analogy is made with the usual caveats.) Then \eqref{eq:Vol.flow} gives us that
\begin{equation}
\frac{\partial}{\partial t}\Vol(\varphi)=\frac{1}{3}\langle \frac{\partial}{\partial t}\varphi,\varphi\rangle_{L^2}
=\frac{1}{3}\langle \d\frac{\partial}{\partial t}\eta,\varphi\rangle_{L^2}
=\frac{1}{3}\langle \frac{\partial}{\partial t}\eta,\d^*_{\varphi}\varphi\rangle_{L^2}.
\end{equation}  
Hence, the gradient flow for $\Vol(\varphi)$ is 
\begin{equation}
\frac{\partial}{\partial t}\eta =\d^*_{\varphi}\varphi \quad\Rightarrow\quad
\frac{\partial}{\partial t}\varphi=\d\frac{\partial}{\partial t}\eta=\d\d^*_{\varphi}\varphi=\Delta_{\varphi}\varphi,
\end{equation}
the Laplacian flow.  (We have ignored the factor $\frac{1}{3}$ which amounts to rescaling time $t$.)
\end{proof}

Proposition \ref{prop:Lflow.gradientflow} immediately yields the following.

\begin{prop}
A closed $\GG_2$ structure on a compact manifold is a critical point of $\Vol(\varphi)$ within a fixed cohomology class if and only if $\varphi$  is torsion-free.
\end{prop}

 We can even say more about the critical points of the volume functional. 
 If we look at the second derivative at a critical point $\varphi(0)$, then by 
\eqref{eq:Vol.flow} along any variation $\varphi(s)=\varphi(0)+\d\eta(s)$ in the cohomology class
\begin{equation}\label{eq:second.var}
\frac{\partial^2}{\partial s^2}\Vol(\varphi)|_{s=0}=\frac{1}{3}\int_M\frac{\partial}{\partial s}\varphi\wedge \frac{\partial}{\partial s}*_{\varphi}\!\varphi|_{s=0},
\end{equation}
since
\begin{equation}
\frac{1}{3}\int_M\frac{\partial^2}{\partial s^2}\varphi\wedge *_{\varphi}\varphi|_{s=0}=\frac{1}{3}\langle \d\frac{\partial^2}{\partial s^2}\eta,*_{\varphi}\varphi\rangle_{L^2}|_{s=0}=\frac{1}{3}\langle\frac{\partial^2}{\partial s^2}\eta,\d^*_{\varphi}\varphi\rangle_{L^2}|_{s=0}=0
\end{equation}
as $\d^*_{\varphi}\varphi=0$ at $s=0$ by assumption.  Now if we write
the variation of $\varphi(s)$ at $s=0$ as in formula \eqref{eq:general.flow} then 
$*_{\varphi}\varphi(s)$ varies by \eqref{eq:starphi} and so we see that
\begin{equation}\label{eq:second.var2}
\int_M\frac{\partial}{\partial s}\varphi\wedge \frac{\partial}{\partial s}*_{\varphi}\!\varphi|_{s=0}=c_0\|f_0\|^2_{L^2}+c_1\|f_1\|^2_{L^2}-\|f_3\|^2_{L^2}
\end{equation}
for some positive constants $c_0$, $c_1$.  If the variation $\varphi(s)$ is orthogonal to the action by diffeomorphisms (meaning that $\frac{\partial\varphi}{\partial s}|_{s=0}$ and the tangent to the diffeomorphism orbit through $\varphi(0)$ are orthogonal), then a slice theorem argument 
forces $f_0=f_1=0$ (see \cite{Hitchin}).  In other words, the slice condition
\begin{equation}
\d^*_{\varphi}\d\eta\in\Omega^2_{14}(M)\quad\Rightarrow\quad \d\eta\in\Omega^3_{27}(M),
\end{equation}
so $f_3=\d\eta(0)$.
Putting this observation together with \eqref{eq:second.var} and \eqref{eq:second.var2}, we see that, orthogonal to the action by diffeomorphisms, we have
\begin{equation}
\frac{\partial^2}{\partial s^2}\Vol(\varphi)|_{s=0}=-\|\d\eta(0)\|_{L^2}^2\leq 0.
\end{equation}

Thus, we have the following.

\begin{thm}\label{thm:Vol.max.closed} Critical points of $\Vol(\varphi)$ on $[\varphi]$ are strict local maxima (modulo the action of diffeomorphisms).  
\end{thm}

This suggests that the gradient flow of the volume functional (i.e.~\eqref{eq:Lflow}) could be well-behaved since its only critical points are maxima.

\subsection{Short-time existence}

A key issue we have avoided in our discussion thus far is the question of whether the Laplacian flow exists or not.  Certainly, if we look at 
\eqref{eq:Lflow} and compare it to \eqref{eq:heat} we would seem to have the wrong sign!  In general, \eqref{eq:Lflow} does not seem to be parabolic in any sense, which is very bad news analytically.

However, again the fact that we are restricting to closed $\GG_2$ structures comes to our rescue.  In this case, we have already seen in \eqref{eq:Lflow.metric} that 
the metric evolves by Ricci flow plus lower order terms and so its flow is parabolic modulo diffeomorphisms.  

If we do DeTurck's trick for the Laplacian flow, using $\d\varphi=0$:
\begin{equation}\label{eq:Lflow.DT}
\frac{\partial}{\partial t}\varphi=\Delta_{\varphi}\varphi+\mathcal{L}_{X(\varphi)}\varphi=\Delta_{\varphi}\varphi+\d(X(\varphi)\lrcorner\varphi),
\end{equation}
 (as in the Ricci flow case, and with the same vector field $X$ given in \eqref{eq:DT.vfield}, in fact) then we might hope that we end up with a genuine parabolic equation in 
 \eqref{eq:Lflow.DT}.  However, this is not the case!
 
 In fact, \eqref{eq:Lflow.DT} is only parabolic in the direction of closed forms, so one has to consider the restricted flow in order to prove short-time existence.  This is a little bit tricky but was done by Bryant--Xu \cite{BryantXu}.  Their paper, which definitely gives a correct result that is fundamental to the subject, has never been published, so we give an account of the proof here, which is essentially the same as in \cite{BryantXu}.
 
\begin{thm}\label{thm:Lflow.short}
 Let $\varphi_0$ be a smooth closed $\GG_2$ structure on a compact manifold $M$.  There exists $\epsilon>0$ so that a unique solution $\varphi(t)$ to the Laplacian flow \eqref{eq:Lflow} with $\varphi(0)=\varphi_0$ and $\d\varphi(t)=0$ exists for all $t\in [0,\epsilon_0]$, where $\epsilon_0$ depends on $\varphi_0$.
\end{thm}
 
 \begin{proof}
We know that if \eqref{eq:Lflow} exists then it will stay in the cohomology class $[\varphi_0]$ by Proposition \ref{prop:Lflow.closed}.  Therefore, we could write 
\begin{equation}
\varphi(t)=\varphi_0+\d\eta(t)
\end{equation}
for a family of exact 3-forms $\d\eta(t)$, and \eqref{eq:Lflow} with the initial condition $\varphi(t)=\varphi_0$ is equivalent to 
\begin{equation}
\frac{\partial}{\partial t}\d\eta =\Delta_{\varphi_0+\d\eta}\d\eta\quad\text{and}\quad \d\eta(0)=0.
\end{equation}

Let $X(\varphi)$ be the vector field given by $X(g_{\varphi})$ in \eqref{eq:DT.vfield}, where we can choose the fixed background metric $k=g_{\varphi_0}$ for example.  Suppose we can solve 
\begin{equation}\label{eq:Lflow.DT2}
\frac{\partial}{\partial t}\d\eta(t)=\Delta_{\varphi_0+\d\eta}\d\eta
+\d(X(\varphi_0+\d\eta)\lrcorner\d\eta)\quad\text{and}\quad \d\eta(0)=0
\end{equation}
uniquely for short time.  Then we can find a unique family of diffeomorphisms $\Phi$ solving \eqref{eq:DT} (since this is the harmonic map heat flow).  Therefore, just as in Proposition \ref{prop.RicciDT}, it follows that $\varphi=\Phi^*(\varphi_0+\d\eta)$ satisfies \eqref{eq:Lflow} with $\varphi(0)=\varphi_0$, $\d\varphi(t)=0$ for short time, and the solution is unique.  (We have used here that the condition for a 3-form to be positive is open, so for $t$ sufficiently small, $\varphi_0+\d\eta(t)$ will be a positive 3-form.) 
We are therefore left to show that \eqref{eq:Lflow.DT2} has a unique short-time solution.  

We know from \eqref{eq:Delta.metric} and 
\cite[Lemma 9.3]{LotayWei.Shi} that
\begin{align}
\Delta_{\varphi}\varphi+\mathcal{L}_{X(\varphi)}\varphi
&=
\frac{1}{2} i_{\varphi}\big(-2\Ric(g_{\varphi})+\mathcal{L}_{X(g_{\varphi})}g_{\varphi}+\frac{2}{21}|\tau_2|^2g_{\varphi}+\frac{1}{4}j_{\varphi}(*_{\varphi}(\tau_2\wedge\tau_2))\big)\nonumber\\&\quad+\frac{1}{2}
\big(\d^*(X(\varphi)\lrcorner\varphi)\big)^{\sharp}\lrcorner*_{\varphi}\!\varphi.\label{eq:Lflow.DT3}
\end{align}
Given that the terms with $\tau_2$ in them in \eqref{eq:Lflow.DT3} are lower order, and the Ricci--DeTurck flow \eqref{eq:RicciDT} is parabolic, it would seem likely that the linearisation of \eqref{eq:Lflow.DT3} is parabolic when restricted to closed forms, and hence that \eqref{eq:Lflow.DT2} is parabolic.  In fact, one may explicitly compute as in \cite{BryantXu} that the linearisation of \eqref{eq:Lflow.DT2} in the direction of exact forms is
\begin{equation}\label{eq:Lflow.DT.lin}
\frac{\partial}{\partial t}\d\zeta=-\Delta_{\varphi_0+\d\eta}\d\zeta+\d\big( Q_{\varphi_0+\d\eta}(\d\zeta)\big),
\end{equation}
where $Q_{\varphi_0+\d\eta}(\d\zeta)$ is order zero in $\d\zeta$ (meaning it depends on just $\d\zeta$ and not its derivatives).   
 Somewhat surprisingly, the sign has switched and \eqref{eq:Lflow.DT.lin} is manifestly parabolic.  

However, we have only shown that \eqref{eq:Lflow.DT2} is parabolic in the direction of exact forms, so we cannot apply standard parabolic theory. Instead we will invoke the Nash--Moser Inverse Function Theorem (see \cite{HamIFT} for a detailed discussion of this theorem).
 
We start by setting up the notation.  We let
\begin{equation}
\mathcal{X}=\d\big(C^{\infty}([0,\epsilon]\times M, \Lambda^2T^*M)\big)\quad\text{and}\quad \mathcal{Y}=\d\Omega^2(M),
\end{equation}
and 
\begin{equation}
\mathcal{U}=\{\d\eta\in\mathcal{X}\,:\,\varphi_0+\d\eta(t)\text{ is a $\GG_2$ structure for all $t$}\},
\end{equation}
which is an open set in $\mathcal{X}$ containing $0$.  We define $\mathcal{F}:\mathcal{U}\to\mathcal{X}\times\mathcal{Y}$ by
\begin{equation}
\mathcal{F}(\d\eta)=\left(\frac{\partial}{\partial t}\d\eta(t)-\Delta_{\varphi_0+\d\eta}\d\eta
-\d(X(\varphi_0+\d\eta)\lrcorner\d\eta),\d\eta(0)\right),
\end{equation}
so that $\mathcal{F}(\d\eta)=(0,0)$ if and only if $\d\eta$ solves \eqref{eq:Lflow.DT2}.  If we can show that $\mathcal{F}$ is locally invertible near $0$, we have that \eqref{eq:Lflow.DT2} has a unique short-time solution and so the proof is complete.

By \eqref{eq:Lflow.DT.lin}, the linearisation of $\mathcal{F}$ at $\d\eta\in\mathcal{U}$ is given by
\begin{equation}
\d\mathcal{F}|_{\d\eta}(\d\zeta)=\left(\frac{\partial}{\partial t}\d\zeta+\Delta_{\varphi_0+\d\eta}\d\zeta-\d\big( Q_{\varphi_0+\d\eta}(\d\zeta)\big),\d\zeta(0)\right).
\end{equation}

Since \eqref{eq:Lflow.DT.lin} is parabolic, we see that if $\d\mathcal{F}|_{\d\eta}(\d\zeta)=(0,0)$ then $\d\zeta=0$ as this is the unique solution to the linear parabolic equation \eqref{eq:Lflow.DT.lin} with zero initial condition.  Further, given $(\d\xi,\d\xi_0)\in\mathcal{X}\times\mathcal{Y}$, we see that $\d\mathcal{F}|_{\d\eta}(\d\zeta)=(\d\xi,\d\xi_0)$ if and only if 
\begin{equation}
\frac{\partial}{\partial t}\d\zeta=-\Delta_{\varphi_0+\d\eta}\d\zeta+\d\big( Q_{\varphi_0+\d\eta}(\d\zeta)\big)+\d\xi\quad\text{and}\quad \d\zeta(0)=\d\xi_0,
\end{equation}
which can be solved by linear parabolic theory.  

We therefore have that $\d\mathcal{F}|_{\d\eta}:\mathcal{X}\to\mathcal{X}\times\mathcal{Y}$ is invertible for all $\d\eta\in\mathcal{U}$.  However, this is not yet enough to show that $\mathcal{F}$ is locally invertible.  Notice that the family of inverses provides a map
$\mathcal{G}:\mathcal{U}\times\mathcal{X}\times\mathcal{Y}\to\mathcal{X}$ 
given by
\begin{equation}
\mathcal{G}(\d\eta,\d\xi,\d\xi_0)=\d\mathcal{F}|_{\d\eta}^{-1}(\d\xi,\d\xi_0).
\end{equation}
To deduce that $\mathcal{F}$ is locally invertible, as we said, we want to invoke the Nash--Moser Inverse Function Theorem which means that we need the following:
\begin{itemize}
\item $\mathcal{X}$ and $\mathcal{Y}$ are tame Fr\'echet spaces and $\mathcal{F}$ is a smooth tame map;
\item $\d\mathcal{F}|_{\d\eta}$ is invertible for all $\d\eta\in\mathcal{U}$ and $\mathcal{G}$ is a smooth tame map.
\end{itemize}

The fact that $C^{\infty}([0,\epsilon]\times M,\Lambda^2T^*M)$ and $\Omega^2(M)$ are naturally tame Fr\'echet spaces is standard, and it therefore quickly follows (from Hodge theory) that $\mathcal{X}$ and $\mathcal{Y}$ are also tame Fr\'echet spaces.  Smooth partial differential operators are smooth tame maps, so $\mathcal{F}$ is a smooth tame map.

The last thing we need to show is that $\mathcal{G}$ is a smooth tame map, 
but this follows immediately from a general result about the family of inverses given by solutions of a smooth family of parabolic partial differential equations, due to Hamilton \cite{Hamilton}.  

Thus, the Nash--Moser Inverse Function Theorem applies to $\mathcal{F}$ and so it is locally invertible as desired.
 \end{proof}
 
 Actually, before DeTurck's trick the same method of Bryant--Xu was used by Hamilton \cite{Hamilton} to prove short-time existence of the Ricci flow.  The reason is that the Ricci flow is not a flow amongst all symmetric 2-tensors really, since the Ricci tensor always satisfies the contracted Bianchi identity.  Therefore, one could consider the flow restricted to those which satisfy this identity.  To prove rigorously that one can do this, one must employ the Nash--Moser Inverse Function Theorem, as Hamilton did.  This is not needed as we have said for the Ricci flow, since DeTurck's trick already removes the issue caused by the Bianchi identity there, but it is needed currently for the Laplacian flow.

\subsection{Results and questions}

\paragraph{Results.} There are several important areas in the Laplacian flow for closed $\GG_2$ structures where progress has been made.
\begin{itemize}
\item Long-time existence criteria based on curvature and torsion estimates along the flow, uniqueness and compactness theory, and real analyticity of the flow \cite{LotayWei.Shi,LotayWei.ra}
\item Stability of the critical points \cite{LotayWei.stab}.
\item Non-collapsing under assumption of bounded torsion \cite{GaoChen}.  
\item Explicit study of the flow in homogeneous situations and other symmetric cases such as nilmanifolds and warped products with a circle  \cite{FFM, FinoRaf, Lauret1, Lauret2, MOV, Nicolini}.
\item Examples and non-existence results for solitons \cite{Lauret1, Lauret2,Lin, LotayWei.Shi}.
\item Eternal solutions for the flow arising from extremally Ricci pinched $\GG_2$ structures \cite{FinoRaf2}.
\item Reduction of the flow to 4 dimensions, with improved long-time existence criteria \cite{FineYao} and analysis of the 4-torus case \cite{HWY}.
\item Reduction of the flow to 3 dimensions, with striking long-time existence and convergence results \cite{LambertL}.
\end{itemize}
It is worth remarking that the scalar curvature here is given by
\begin{equation}
R(g_{\varphi})=-\frac{1}{2}|\tau_2|^2,
\end{equation}
so having a bound on torsion is equivalent to a bound on scalar curvature.

\paragraph{Questions.} There are many open problems in the area.
\begin{itemize}
\item Does the flow exist as long as the torsion is bounded?
\item Can a volume bound be used to control the flow?
\item Are there any compact examples which develop a singularity in finite time?
\item Is there a relationship between the flow and calibrated submanifolds, specifically coassociative submanifolds?
\end{itemize}
For the last two points, there is an example due to Bryant \cite{Bryant} which shows that singularities can happen at infinite time (i.e.~the flow exists for all time but does not converge), and that the singularity is related to coassociative geometry.  

We can also ask whether the Laplacian flow is potentially useful to study other classes of $\GG_2$ structures.  For example, naively if we assume
$\varphi$ is coclosed and the flow exists then it should stay coclosed since then
\begin{equation}
\Delta_{\varphi}\varphi=\d^*_{\varphi}\d\varphi.
\end{equation}
So, an obvious question is: does this flow exist?  Gavin Ball has informed the author that, in general, the Laplacian flow will \emph{not} preserve the coclosed condition, so the answer would appear to be negative.

\section{Laplacian coflow}

Another approach to studying $\GG_2$ structures was introduced in \cite{KarMcKayTsui}. 

\begin{dfn}
The Laplacian coflow for $\GG_2$ structures is given by:
\begin{equation}\label{eq:Lcoflow}
\frac{\partial}{\partial t}*_{\varphi}\!\varphi=\Delta_{*_{\varphi}\varphi}\!*_{\varphi}\!\varphi,
\end{equation}
where $\Delta_{*_{\varphi}\varphi}$ is the Hodge Laplacian of the metric determined by $*_{\varphi}\varphi$.
(Actually, in \cite{KarMcKayTsui}, they introduced \eqref{eq:Lcoflow} with a minus sign on the right-hand side by analogy with the heat equation \eqref{eq:heat}, but as we shall see below the ``correct'' sign in the equation is that given, just as in \eqref{eq:Lflow}.)  
\end{dfn}

\noindent 
Here one has to be a little careful since the 4-form $*_{\varphi}\varphi$ is not quite equivalent to the 3-form $\varphi$.  In particular, the 4-form does not determine the orientation, but we can assume we have an initial orientation which stays fixed along the flow. 

 Again by integration by parts it is easy to see that on a compact manifold
\begin{equation}
\Delta_{*_{\varphi}\varphi}\!*_{\varphi}\!\varphi=0\quad\Leftrightarrow\quad
\d\varphi=\d^*_{\varphi}\varphi=0,
\end{equation}
so the critical points are again the torsion-free $\GG_2$ structures.

\subsection{Coclosed $\GG_2$ structures}

The proposal in \cite{KarMcKayTsui} is to restrict \eqref{eq:Lcoflow} to 
closed 4-forms (so coclosed $\GG_2$ structures).  If the flow exists, meaning it preserves closed forms as in the Laplacian flow setting, we would have that:
\begin{equation}
\frac{\partial}{\partial t}*_{\varphi}\!\varphi=\d\d^*_{\varphi}\!*_{\varphi}\!\varphi.
\end{equation}
Thus, again, the flow will stay in the given cohomology class $[*_{\varphi}\varphi(0)]$ as long as it exists.  Therefore, the Laplacian coflow can be seen as a possible means to deform $*_{\varphi}\varphi$ in its cohomology class so that it becomes torsion-free.

\begin{prop}
If $*_{\varphi}\varphi(t)$ satisfies \eqref{eq:Lcoflow} on $M$ and 
$\d\!*_{\varphi}\!\varphi(t)=0$ then
$*_{\varphi}\varphi\in[*_{\varphi}\varphi(0)]\in H^4(M)$ for all $t$ for which the flow exists.
\end{prop}

We can easily modify our discussion of the volume functional $\Vol(\varphi)$, given in \eqref{eq:Vol}, and the Laplacian flow in Proposition \ref{prop:Lflow.gradientflow} and Theorem \ref{thm:Vol.max.closed} to show the following. 

\begin{thm} The flow \eqref{eq:Lcoflow} for coclosed $\GG_2$ structures $*_\varphi\varphi$ is the gradient 
flow of the volume functional in \eqref{eq:Vol} restricted to $[*_{\varphi}\varphi]$ and the  critical points are strict local maxima for the volume functional (modulo diffeomorphisms).  
\end{thm}

\noindent At this point, things are looking quite good in the study of the Laplacian coflow.

However, it is now worth going back to the earlier discussion of the coclosed condition \eqref{eq:coclosed}.  Taking a positive interpretation of the h-principle result (Theorem \ref{thm:hprinciple}), we can always assume that our $\GG_2$ structure is coclosed and therefore \eqref{eq:Lcoflow} potentially allows us to study the space of all $\GG_2$ structures whilst restricting to the coclosed ones.  A more negative outlook is to say that studying the Laplacian coflow means that we are effectively studying all $\GG_2$ structures, which only has topological rather than geometric content, and so perhaps it is unrealistic to expect the flow to be well-behaved in such generality.

\subsection{Short-time existence: a modification}

Given this discussion it is worth tackling the key issue of whether the flow \eqref{eq:Lcoflow} even exists.  To tackle this, we would employ the DeTurck trick and hope to show that \eqref{eq:Lcoflow} is parabolic, possibly only in the direction of closed forms, just as in the case of \eqref{eq:Lflow} as in Theorem \ref{thm:Lflow.short}.  Unfortunately, this does not work!  Therefore, we cannot currently say (regardless of which sign we choose in \eqref{eq:Lcoflow}) that the Laplacian coflow even exists.  Nonetheless, one can find solutions to it in special cases \cite{KarMcKayTsui}, so we can continue to ask the question: does \eqref{eq:Lcoflow} exist restricted to coclosed $\GG_2$ structures?

An approach taken in \cite{Grig1} is to modify \eqref{eq:Lcoflow} to get a family of flows depending on an arbitrary constant $c$.  

\begin{dfn}
The modified Laplacian coflow(s) (recalling the torsion forms in \eqref{eq:torsion}) for $c\in\R$ is defined by
\begin{equation}\label{eq:Lcoflow.modified}
\frac{\partial}{\partial t}*_{\varphi}\!\varphi=\Delta_{*_{\varphi}\varphi}\!*_{\varphi}\!\varphi+\d\left( \left(c-\frac{7}{2}\tau_0\right)\varphi\right),
\end{equation}
again restricted to coclosed $\GG_2$ structures.
\end{dfn}

This again moves in the cohomology class and has the added benefit of defining a flow which is parabolic in the direction of closed forms (modulo diffeomorphisms), and therefore \eqref{eq:Lcoflow.modified} is guaranteed to exist on a compact $7$-manifold by essentially the same proof as Theorem \ref{thm:Lflow.short}.

\begin{thm}\label{thm:Lcoflow.short}
 Let $\varphi_0$ be a smooth coclosed $\GG_2$ structure on a compact manifold $M$.  There exists $\epsilon>0$ so that a unique solution $*_{\varphi}\varphi(t)$ to the modified Laplacian coflow \eqref{eq:Lcoflow.modified} with $*_{\varphi}\varphi(0)=*_{\varphi_0}\varphi_0$ and $\d*_{\varphi}\varphi(t)=0$ exists for all $t\in [0,\epsilon]$.
\end{thm}

\paragraph{Critical points.} The short-time existence is of course important, but \eqref{eq:Lcoflow.modified} is no longer obviously a gradient flow and there is also no reason why the critical points of the flow are torsion-free $\GG_2$ structures.  It is clear that if $\varphi$ is torsion-free then the right-hand side of \eqref{eq:Lcoflow.modified} vanishes (since $\tau_0=0$ and $\d\varphi=0$).  However, suppose we choose a nearly parallel $\GG_2$ structure $\varphi$ as in \eqref{eq:nearly.parallel},  recalling that $\tau_0$ is constant in that case.  Then we see that
\begin{equation}
\Delta_{*_\varphi\varphi}\!*_{\varphi}\!\varphi=\d\d^*_{\varphi}\!*_{\varphi}\!\varphi=
\d\!*_{\varphi}\!\d\varphi=\d\!*_{\varphi}\!\tau_0\!*_{\varphi}\!\varphi=\tau_0\d\varphi=\tau_0^2\!*_{\varphi}\!\varphi
\end{equation}
and
\begin{equation}
\d\left(\left(c-\frac{7}{2}\tau_0\right)\varphi\right)=\tau_0\left(c-\frac{7}{2}\tau_0\right)*_{\varphi}\!\varphi.
\end{equation}
Hence,
\begin{equation}
\Delta_{*_{\varphi}\varphi}\!*_{\varphi}\!\varphi+\d\left( \left(c-\frac{7}{2}\tau_0\right)\varphi\right)=\tau_0\left(c-\frac{5}{2}\tau_0\right)\!*_{\varphi}\!\varphi,
\end{equation}
which vanishes if $c=\frac{5}{2}\tau_0$.  Therefore, we will also get certain nearly parallel $\GG_2$ structures as critical points.  Notice here that we can change $\tau_0$ by rescaling our nearly parallel $\GG_2$ structure, so the flow will distinguish a certain scale for the nearly parallel $\GG_2$ structures.  For example, the 7-sphere has a canonical 
nearly parallel $\GG_2$ structure, and only the 7-sphere of a certain size (depending on a choice of positive $c$) will be a critical point whereas others will not.  

Altogether, this is a rather strange situation, which shows that the 
modified Laplacian coflow, though parabolic, has some potentially undesirable properties.

\subsection{Results and questions}

The Laplacian coflow \eqref{eq:Lcoflow} and its modification \eqref{eq:Lcoflow.modified} have so far received rather little attention, but the key results in the area include the following.
\begin{itemize}
\item Soliton solutions arising from warped products and symmetries \cite{Grig2, KarMcKayTsui}.
\item Explicit study of the flow for symmetric situations \cite{FFB, FB, MOV}.
\item Long-time existence criteria based on curvature and torsion estimates along the modified Laplacian coflow, and non-collapsing for \eqref{eq:Lcoflow.modified} 
under assumption of bounded scalar curvature \cite{GaoChen}. 
\end{itemize}

There are also many open problems.
\begin{itemize}
\item Does the Laplacian coflow exist (possibly under some further assumptions)?
\item What are the critical points of the modified Laplacian coflow?
\item What do dimensional reductions of the Laplacian coflow  or its modification look like?
\end{itemize}

\section{Dirichlet energies and spinorial flows}

The final set of flows we will discuss are explicitly constructed as gradient flows.  We have already seen gradient flows of the volume functional \eqref{eq:Vol} for $\GG_2$ structures (when restricted to closed or coclosed $\GG_2$ structures) but these are actually atypical examples of gradient flows.  The reason is that $\Vol(\varphi)$ is of order $0$ in $\varphi$.  Typically, gradient flows are with respect to functionals which are first order in the geometric quantity in question.  For example, mean curvature flow is the gradient flow for the volume functional which is first order in the immersion.  

\subsection{Dirichlet energies}

In light of \eqref{eq:torsion.free}, there are two clear choices for possible functionals on a compact 7-manifold to consider.

\begin{dfn} We let
\begin{equation}\label{eq:Dirichlet.functionals}
\mathcal{C}(\varphi)=\frac{1}{2}\int_M |\nabla_{\varphi}\varphi|_{g_\varphi}^2\vol_{\varphi}\geq 0
\quad\text{and}\quad \mathcal{D}(\varphi)= \frac{1}{2}\int_M|\d\varphi|^2_{g_\varphi}+|\d^*_{\varphi}\varphi|^2_{g_\varphi}\vol_{\varphi}\geq 0.
\end{equation}
Observe  that $\mathcal{D}$ and $\mathcal{C}$ differ by the total scalar curvature functional (see \cite{WW2}):
\begin{equation}
\mathcal{D}(\varphi)-\mathcal{C}(\varphi)=\int_MR(g_{\varphi})\vol_{\varphi}.
\end{equation}
\end{dfn}

The functionals $\mathcal{C}$ and $\mathcal{D}$ are both 
what one might call ``Dirichlet energies'', by analogy with \eqref{eq:Dirichlet}.  Therefore, one could potentially call the gradient flows of these quantities ``heat flows''.  There is also the possibility to generalise slightly and consider (in the notation of \eqref{eq:torsion}):
\begin{equation}\label{eq:Denergies}
\mathcal{D}_{\nu}(\varphi)=\sum_i \frac{\nu_i}{2}\int_M|\tau_i|^2_{g_{\varphi}}\vol_{\varphi}\geq 0
\end{equation}
for positive constants $\nu_i$: this encompasses $\mathcal{C}$  and $\mathcal{D}$ for appropriate choices of $\nu_i$.

All of these functionals are considered in \cite{WW1,WW2} where the authors show the following.

\begin{thm} The critical points of the Dirichlet energies $\mathcal{D}_{\nu}$ in \eqref{eq:Denergies} are the torsion-free $\GG_2$ structures, which are the absolute minimizers for the functionals (since they are precisely zero at these points).  
\end{thm}

\noindent They have also shown short-time existence of the gradient flows of the functionals $\mathcal{D}_{\nu}$, since they are parabolic (modulo diffeomorphisms) and so the standard DeTurck's trick approach can be used.

The key main results are the following (again in \cite{WW1,WW2}).
\begin{itemize}
\item Stability of the critical points.
\item Some simple examples of solitons.
\end{itemize}

In \cite{WW2} the following observation is made.

\begin{prop} The volume functional $\Vol(\varphi)$ in \eqref{eq:Vol} is monotone decreasing along the gradient flow of $\mathcal{C}$ in \eqref{eq:Dirichlet.functionals}.  In fact, it is a convex function along the flow.
\end{prop}

\noindent  This is contrary to our earlier results for the Laplacian flow and coflow which viewed $\Vol(\varphi)$ as having strict maxima (modulo diffeomorphisms).  In our view, this indicates a key drawback in this Dirichlet energy approach which we shall return to later.

\subsection{Spinorial flow}

Up to a constant multiplicative factor, the functional $\mathcal{C}$ can also be written in terms of unit spinors $\sigma$ on a 7-manifold as
\begin{equation}\label{eq:spinor.functional}
\mathcal{E}(\sigma)=\frac{1}{2}\int_M|\nabla_{g}\sigma|_{g}^2\vol_{\sigma}\geq  0.
\end{equation}
This formulation of the Dirichlet energy is something which can clearly be extended beyond $\GG_2$ geometry.

\begin{dfn}
Define the following functional on pairs of metrics and unit spinors on a compact oriented spin manifold $M$:
\begin{equation}\label{eq:spinor.functional.modified}
\mathcal{E}(g,\sigma)=\frac{1}{2}\int_M|\nabla_g\sigma|_g\vol_g\geq 0.
\end{equation}
The gradient flow of $\mathcal{E}$ is called the spinorial or spinor flow.
\end{dfn}

\noindent Here, unlike the $\GG_2$ case, one can vary the metric and spinor independently, with the caveat that the spinor must remain a $g$-spinor and be unit length.

In \cite{AWW} the authors show that the critical points of $\mathcal{E}(g,\sigma)$, when the dimension of the manifold $M$ is at least $3$, are given by parallel unit spinors $\sigma$ and so the metric $g$ is Ricci-flat of special holonomy.  They also show that the associated gradient flow (the spinorial flow) is parabolic modulo diffeomorphisms and so has short-time existence.  This time the analysis is more involved because the space of unit spinors varies as the metric varies, but despite these complications the final result is as one would expect.

\begin{itemize}
\item All of the ingredients enable one to prove that the critical points of
 \eqref{eq:spinor.functional.modified} are stable under its gradient flow, as shown in \cite{Lothar}. 
 \item The special case of the spinorial flow on Berger spheres is studied in detail in \cite{Wittmann}. 
 \item The full analysis of the 2-dimensional case, which has special features not covered in \cite{AWW}, is part of work in progress at the  time of writing.
\end{itemize}
 
\subsection{Questions}

\paragraph{Morse functionals.} The motivation for studying gradient flows of functionals is that one would hope (at least formally) that the functional is Morse (or Morse--Bott) on the space of geometric objects in question.  Therefore, the critical points would be encoded by the topology of the space of geometric objects.  Very often making this formal picture rigorous is very challenging, but still motivational.  For example, in the study of  surfaces in 3-manifolds (or, more generally, hypersurfaces in $n$-manifolds), since the topology of the space of such surfaces is infinite and the volume functional is (in some sense) formally a Morse function on this space, one might hope to construct infinitely many minimal surfaces in any 3-manifold: this is a conjecture due to Yau from 1982, which has been proved when the Ricci curvature of the 3-manifold is positive or generic \cite{IMNYau, MarquesNeves}, and recently claimed (at the time of writing) for all metrics on 3-manifolds in \cite{SongYau}.

Here, the Dirichlet functionals are defined on the space of $\GG_2$ structures (modulo diffeomorphism).  However, unlike the case of the volume functional on hypersurfaces, the only critical points of the functionals are absolute minimizers.  Therefore, if the functional is a Morse function then the best we can hope for is that the space of $\GG_2$ structures on our given manifold could be contractible onto the torsion-free $\GG_2$ structures.    

\paragraph{Scaling.} More than that, just as we saw with the volume functional in \eqref{eq:Vol}, the best way to reduce the Dirichlet energy is to send the $3$-form to zero by scaling, which is clearly useless.  The same thing of course happens when studying hypersurfaces under mean curvature flow, but we can stop the hypersurface from being contracted to a point by simply choosing a nontrivial homology class for our initial hypersurface, which then obviously cannot contain the ``zero'' hypersurface.  The same happens in the Laplacian flow and coflow: the cohomology class is fixed so that one kills the action of rescaling.  

Unfortunately, when studying the Dirichlet energies, the class of $\GG_2$ structure is not preserved and so all one can do is look at the homotopy class of the initial $\GG_2$ structure $\varphi$: this class is always homotopic to $0$ in the space of $\GG_2$ structures just by rescaling (although 0 is, of course, not a $\GG_2$ structure).  Therefore, one might expect for generic initial conditions that the Dirichlet energy gradient flows just send the 3-form to 0, which is certainly an absolute minimizer of the energy, but does not appear to provide any meaningful content.

This discussion leads to the following question: is there a way to modify or restrict the gradient flows of \eqref{eq:Dirichlet.functionals} to ensure that the 3-form does not go to 0?

\section{Conclusions}

The study of geometric flows of $\GG_2$ structures has seen some important progress, but it is fair to say that at the time of writing the subject is still in its relative infancy.  The flows we have described have both pros and cons, and seek to tackle different problems, so it is potentially interesting to further investigate all of them to see what we can learn about $\GG_2$ structures.  There are also potentially further flows of $\GG_2$ structures that could be useful, such as the flow of isometric $\GG_2$ structures introduced in \cite{Bagaglini}.  The field is clearly vibrant and wide open for discovery and progress.

In particular, it we would be very exciting if by studying geometric flows we can uncover a new criteria (geometric or topological) for the existence or otherwise of torsion-free $\GG_2$ structures.  Whilst this is an ambitious goal, by seeking to solve it we may well acquire a much better understanding of the space of $\GG_2$ structures.

\end{document}